\documentclass[11pt]{article}
\usepackage{amssymb,latexsym}
\usepackage{graphicx,amsmath,amsfonts,amsthm}
\usepackage{comment}
\usepackage{hyperref}

\def\Q{\mathbb{Q}}
\def\R{\mathbb{R}}
\def\Z{\mathbb{Z}}

\def\row#1#2{{#1}_1,\ldots ,{#1}_{#2}}
\def\brow#1#2{{\bf {#1}}_1,\ldots ,{\bf {#1}}_{#2}}

\def\srow#1#2{{#1}_2,,\ldots ,{#1}_{#2}}

\def\blcomb#1#2#3{{#1}_1{\bf {#2}}_1+\cdots +{#1}_{#3}{\bf {#2}}_{#3}}
\def\2vec#1#2{\left(\begin{array}{c}{#1}\\{#2}\end{array}\right)}
\def\mod#1{\ \hbox{\rm (mod $#1$)}}

\def\blcomb#1#2#3{{#1}_1{\bf {#2}}_1+{#1}_2{\bf {#2}}_2+\cdots +{#1}_{#3}{\bf {#2}}_{#3}}

\def\twomat#1#2#3#4{\left(\begin{array}{rr}
{#1} & {#2}\\ {#3} & {#4}\end{array}\right)}

\def\blcomb#1#2#3{{#1}_1{\bf {#2}}_1+{#1}_2{\bf {#2}}_2+\cdots +{#1}_{#3}{\bf {#2}}_{#3}}

\def\twomat#1#2#3#4{\left(\begin{array}{cc}
{#1} & {#2}\\ {#3} & {#4}\end{array}\right)}
\newtheorem{theorem}{Theorem}
\newtheorem{corollary}{Corollary}
\newtheorem{lemma}{Lemma}

\newtheorem{proposition}{Proposition}
\newtheorem{example}{Example}

\newtheorem{definition}{Definition}
\newtheorem{problem}{Problem}

\title{\bf Weighted and Roughly Weighted Simple Games}
\author{\bf Tatiana Gvozdeva and Arkadii Slinko}
\date{\today}

\begin{document}

\maketitle

\begin{abstract}
This paper contributes to the program of numerical characterisation and classification of simple games outlined in the classical monograph of  von Neumann and Morgenstern \cite{vNM}. 
One of the most fundamental questions of this program is what makes a simple game a weighted majority game. The necessary and sufficient conditions that guarantee weightedness
were obtained by Elgot \cite{Elgot} and refined by Taylor and Zwicker \cite{TZ92}.
If a simple game does not have
weights, then rough weights may serve as a reasonable substitute
(see their use in Taylor and Zwicker \cite{TZ}). A simple game  is  roughly weighted if there exists a system of weights and a threshold such that all coalitions whose combined weight is above the threshold are winning and all coalitions whose combined weight is below the threshold are
losing and a tie-breaking is needed to classify the coalitions whose combined weight is exactly the threshold. Not all simple games are roughly weighted, and the class of projective games \cite{Rich} is a prime example.

In this paper we give necessary and sufficient conditions for a simple game to have rough weights. We define two functions $f(n)$ and $g(n)$ that measure the deviation of a simple game from a weighted majority game and roughly weighted majority game, respectively. We formulate known results in terms of lower and upper bounds for these functions and improve those bounds.
We also investigate rough weightedness of simle games with a small number of players.
\end{abstract}

\section{Introduction}
\label{intro}

In the classical book \cite{vNM} von Neumann and Morgenstern outlined the programme of numerical classification and characterisation of all simple games.\footnote{See Section 50.2.1, page 433 of the Third Edition.} They viewed the introduction of weighted majority games as the first step in this direction.  They noted however\footnote{Section 5.3 of the same book} that already for six players not all games have weighted majority representation and they also noted that for seven players some games do not have weighted majority representation in a much stronger sense. 
Therefore one of the most fundamental questions of this programme is to find out what makes a simple game a weighted majority game. The next step is to measure the deviation of an arbitrary game from a weighted majority game in terms of a certain function $f(n)$ of the number of players $n$ and to obtain lower and upper bounds for this function. 

The necessary and sufficient conditions that guarantee weightedness of a game 
are known. Elgot \cite{Elgot} obtained them in terms of asummability. Taylor and Zwicker \cite{TZ92} obtained necessary and sufficient conditions later but independently in terms of trading transforms.  The advantage of the latter characterisation is that it is constructive in the sense that only finitely many conditions (which depends on the number of players) has to be checked to decide if the game is weighted or not. More precisely, they showed that a simple game is
weighted majority game if no sequence of winning coalitions up to the  length $2^{2^n}$ can be converted into a sequence of losing
coalitions by exchanging players. 

The sequence of coalitions
\begin{equation}
\label{Tt}
{\cal T}=(\row Xj;\row Yj)
\end{equation}
is called a trading transform if the coalitions $\row Xj$ can be converted into the coalitions $\row Yj$ by rearranging players. If game $G$ with $n$ players does not have weights, then the characterisation of Taylor and Zwicker implies that there exists a trading transform (\ref{Tt}) where all $\row Xj$ are winning and all $\row Yj$ are losing. We call such a trading transform a certificate of non-weightedness. We denote the minimal length $j$ of such a certificate ${\cal T}$ as $f(G)$ and set $f(n)=\max_G f(G)$, where $G$ runs over all games without weighted majority representation. A discussion about the nature of this function is in order. If $f(G)=2$, then the game $G$ is extremely non-weighted, in fact most games, as noted in \cite{TZ}, are of this kind. And it is easy to find a certificate of their non-weightedness. Games with the condition $f(G)>2$ behave in some respects as weighted games, in particular the desirability relation on singletons is a weak order (this is why they are called complete games \cite{CF,TZ}). For such games Carreras and Freixas  \cite{CF,TZ,FM} obtained a useful classification result. So the larger $f(G)$ the closer the game $G$ to weighted majority games. And it is not surprising that it gets more and more difficult to find a certificate of their non-weightedness. It is important to know what is the maximal length of certificates that has to be checked in order to declare that $G$ is weighted. The function $f(n)$ shows exactly this length. This is a complete analogue to the Fishburn's function $f(n)$ defined for linear qualitative probability orders \cite{PF1}.

Many old results can be nicely expressed in terms of this function. In particular the results of Taylor and Zwicker \cite{TZ92,TZ95} (and earlier Gabelman \cite{Gabelman} for small values of $n$) can be presented as  lower and upper bounds for $f(n)$ as follows
\begin{equation}
\label{theirbounds}
\lfloor \sqrt{n}\rfloor \le f(n)\le 2^{2^n}. 
\end{equation}
In this paper we improve both bounds as follows:
\begin{equation}
\label{ourbounds1}
\left\lfloor \frac{n-1}{2}\right\rfloor \le f(n)\le (n+1)2^{\frac{1}{2}n \log_2 n }.
\end{equation}

If a simple game does not have weights, then rough weights may serve as a reasonable substitute (see \cite{TZ}). 
The idea is the same as in the use of tie-breaking in voting in case when only one alternative is to be elected. If the combined weight of a coalition is greater than a certain threshold, then it is winning, if the combined weight is smaller than the threshold, then this coalition is losing. If its weight is exactly the threshold, then it can go either way depending on the ``tie-breaking'' rule. 

In this paper we obtain a necessary and sufficient conditions for the existence of rough weights. 
We prove that a game $G$ is roughly weighted majority game if for no $j$ smaller than ${(n+1)2^{\frac{1}{2}n \log_2 n }}$ there exist a certificate of non-weightedness of length $j$ with the grand coalition among winning coalitions and the empty coalition among losing coalitions. Let us call such certificates {\em potent}. For a game $G$ without rough weights we define by $g(G)$ the lengths of the shortest potent certificate of non-weightedness. Then, a function $g(n)$ can be naturally defined which is fully analogous to $f(n)$. It shows the maximal length of  potent certificates that has to be checked in order to decide if $G$ is roughly weighted or not. One of the main result of this paper can be formulated in terms of function $g(n)$ as follows:
\begin{equation}
\label{ourbounds2}
2n+3 \leq g(n)<(n+1)2^{\frac{1}{2}n \log_2 n }.
\end{equation}
We prove the lower bound by constructing examples. 

We also study rough weightedness of small games. We show that all games with $n\le 4$ players, all strong and proper games with $n\le 5$ players and all constant-sum games with $n\le 6$ players are roughly weighted. Thus the smallest constant sum game that is not roughly weighted is the game with seven players obtained from the Fano plane \cite{vNM}, page 470. This game is the smallest representative of the class of projective games \cite{Rich}. One of the consequences of our characterisation is that all projective games do not have rough weights.


\section{Definitions and examples}
\label{preliminaries}

Let us consider a finite set $P$ consisting of $n$ elements which we will call players.  For convenience $P$ can be taken to be $[n]=\{1,2,\ldots,n\}$.

\begin{definition}
A simple game is a pair $G=(P,W)$, where $W$ is a subset of the power set $2^P$  which satisfies the monotonicity condition: 
\begin{quote}
if $X\in W$ and $X\subset Y\subseteq P$, then $Y\in W$.
\end{quote}
We also require that $W$ is different from $\emptyset$ and $P$ (non-triviality assumption).
\end{definition}

Elements of the set $W$ are called {\em winning coalitions}. We also define the set $L=2^P\setminus W$ and call elements of this set {\em losing coalitions}. A winning coalition is said to be {\em minimal} if every its proper subset is a losing coalition. Due to monotonicity, every simple game is fully determined by the set of its minimal winning coalitions.\par
\medskip

For $X\subseteq P$ we will denote its complement $P\setminus X$ as $X^c$.

\begin{definition}
A simple game is called proper if $X\in W$ implies that $X^c \in L$ and strong if $X\in L$ implies that $X^c \in W$. A simple game which is proper and strong is called a constant sum game.
\end{definition}

In a constant sum game there are exactly $2^{n-1}$ winning coalitions and exactly the same number of losing ones.

\begin{definition}
A simple game $G$ is called {\em weighted majority game} if there exist non-negative reals $\row wn$, and a positive real number $q$, called quota, such that $X\in W$ iff $\sum_{i\in X} w_i\ge q$. Such game is denoted $[q;\row wn]$. We also call $[q;\row wn]$ as a {\em voting representation} for $G$.
\end{definition}




\begin{example}[\cite{FM}]
The UN Security Council consists of five permanent and 10 non-permanent countries. A passage requires approval of at least nine countries, subject to a veto by any one of the permanent members. This is a weighted simple game with a voting representation
\[
[39;7,7,7,7,7,1,1,1,1,1,1,1,1,1,1].
\]
\end{example}

Other interesting examples of simple games can be found in \cite{TZ, FM}. \par\medskip

As von Neumann and Morgenstern showed \cite{vNM} all simple games with less than four, every proper or strong simple game with less than five and every constant sum game with less than six players has a voting representation. For six players the situation is different and there are constant sum games with six players that are not weighted \cite{vNM}. 

\begin{example}
\label{n=6}
Let $n=6$. Let us include in $W$ all sets of cardinality four or greater, $22$ sets in total. We want to construct a proper game, therefore we have to choose and include in $W$ at most one set out of each of the 10 pairs $(X,X^c)$, where $X$ is a subset of cardinality three. Suppose we included sets  $X_1=\{1,2,4\}$, $X_2=\{1,3,6\}$, $X_3=\{2,3,5\}$, $X_4=\{1,4,5\}$, $X_5=\{2,5,6\}$, $X_6=\{3,4,6\}$ in $W$ (and four other 3-element sets to insure that the game is constant sum).  If this game had a voting representation $[q; \row w6]$, then the following system of inequalities must have a solution:
\begin{equation}
\label{1...6}
\sum_{i\in X_j}w_i>\sum_{i\in X_j^c}w_i,\qquad j=1,\ldots,6.
\end{equation}
This system is nevertheless inconsistent. 
\end{example}

However if we convert all six inequalities (\ref{1...6}) into equalities, then there will be a 1-dimensional solution space spanned by $(1,1,1,1,1,1)$ which shows that this game ``almost'' have a voting representation $[3; 1,1,1,1,1,1]$. Indeed, if we assign weight 1 to every player, then all coalition whose weight falls below the threshold 3 are in $L$, all coalitions whose total weight exceeds this threshold are in $W$. However, if a coalition has total weight of three, i.e. it is equal to the threshold, it can be either winning or losing. 

\begin{definition}[ \cite{TZ}, p.78]
A simple game $G$ is called {\em roughly weighted} if there exist non-negative real numbers $\row wn$ and a  real number $q$, called quota, not all equal to zero, such that for $X\in 2^P$ the condition $\sum_{i\in X} w_i< q$ implies $X\in L$, and $\sum_{i\in X} w_i> q$ implies $X\in W$.  We say that $[q;\row wn]$ is  a {\em rough voting representation} for $G$.
\end{definition}

The simple game in Example~\ref{n=6} is roughly weighted with  a rough voting representation $[3;1,1,1,1,1,1]$. We will show later (Theorem~\ref{rough_weights_for 6}) that any constant sum game with six players has a rough voting representation.
 In threshold logic roughly weighted games  correspond to pseudo-threshold functions (see \cite{Mu}, p.208).

\begin{example}[Fano plane game \cite{vNM}]
\label{projective}
Let us denote $P=[7]$ the set of points  of the projective plane of order two, called the Fano plane. Let $P$ be the set of players of the new game. Let us also take the seven lines of this projective plane as minimal winning coalitions:
\begin{equation}
\label{lines}
\{1,2,3\},\ \{3,4,5\},\ \{1,5,6\},\ \{1,4,7\},\ \{2,5,7\},\ \{3,6,7\},\ \{2,4,6\}.
\end{equation}
We will denote them by $\row X7$, respectively.
This, as it is easy to check, defines a constant sum game, which we will denote $Fano$. If it had a rough voting representation $[q; \row w7]$, then the following system of inequalities will be consistent:
\begin{equation}
\label{1...7}
\sum_{i\in X_j}w_i\ge\sum_{i\in X_j^c}w_i,\qquad j=1,\ldots,7.
\end{equation}
However adding all the equations up we get $\sum_{i=1}^7w_i\le 0$ which shows that this system does not have solutions with non-negative coordinates other than the zero solution. Since all weights are equal to zero, by the definition, the threshold must non-zero, coalitions (\ref{lines}) cannot be winning. Hence this simple game is not roughly weighted.
\end{example}



\section{Games and Ideals}

Firstly we would like to redefine trading transforms algebraically.  Let $T=\{-1,0,1\}$ and $T^n=T\times T\times \ldots T$ ($n$ times). With any pair $(X,Y)$ of subsets $X,Y\in [n]$ we define
\[
{\bf v}_{X,Y}=\chi(X)-\chi(Y)\in T^n,
\]
where $\chi(X)$ and $\chi(Y)$ are the characteristic vectors of subsets $X$ and $Y$, respectively. 

Let now $G=(P,W)$ be a game. We will associate an algebraic object with $G$.  For any pair $(X,Y)$, where $X$ is winning and $Y$ is losing, we put in correspondence the following vector
$
{\bf v}_{X,Y}.
$
The set of all such vectors we will denote $I(G)$.

\begin{definition}
Let ${\bf e}_i=(0,\ldots, 1,\ldots,0)$, where the only nonzero element 1 is in the $i$th position. Then a subset $I\subseteq T^n$ will be called an {\em ideal} in $T^n$ if for any $i=1,2,\ldots, n$
\begin{equation}
\label{plus_e_i}
({\bf v}\in I\ \text{and ${\bf v}+{\bf e}_i\in T^n) \Longrightarrow {\bf v}+{\bf e}_i\in I$}.
\end{equation}
\end{definition}

\begin{proposition}
\label{I(G)}
Let $G$ be a game with $n$ players. Then $I(G)$ is an ideal in $T^n$.
\end{proposition}

\begin{proof}
The condition (\ref{plus_e_i}) follows directly from the monotonicity condition for games. Indeed, if ${\bf v}_{X,Y}+{\bf e}_i$ is in $T^n$, then this amounts to either addition of $i$ to $X$, which was not there, or removal of $i$ from $Y$. Both operations maintain $X$ winning and $Y$ losing.
\end{proof}

We note that $G$ can be uniquely recovered from $I(G)$ only for proper games. The key to this recovery is to consider all vectors from $I(G)$ without zeros. Indeed, if $X$ is winning coalition, then $X^c$ is losing and ${\bf v}_{X,X^c}\in I(G)$. This vector does not contain zeros and $X$ can be recovered from it uniquely. 

\begin{proposition}
\label{w&rw}
Let $G$ be a finite simple game. Then:
\begin{enumerate}
\item[(a)] $G$ is weighted iff the system
\begin{equation}
\label{inequality>}
{\bf v}\cdot {\bf x}>0,\qquad {\bf v}\in I(G)
\end{equation}
has a solution.
\item[(b)] $G$ is roughly weighted iff the system
\begin{equation}
\label{inequalityge}
{\bf v}\cdot {\bf x}\ge 0,\qquad {\bf v}\in I(G)
\end{equation}
has a non-zero  solution.
\end{enumerate}
\end{proposition}

\begin{proof}
The proof of (a) is contained in \cite{TZ} (see Lemma~2.6.5 and comment on page 6 why all weights can be chosen non-negative).  

Let us prove (b). 
Suppose $G$ is roughly weighted,  Let ${\bf v}={\bf v}_{X,Y}\in I(G)$. Then $X\in W$, $ Y\in L$ and  ${\bf w}=(\row wn)$ satisfies 
\begin{equation}
\label{geqge}
\sum_{t\in X}w_t \ge q\ge \sum_{s\in Y}w_s.
\end{equation}
This implies $\sum_{t\in X}w_t -  \sum_{s\in Y}w_s\ge 0$ or ${\bf v}\cdot {\bf w}\ge 0$ and 
then ${\bf w}$  is a non-zero solution of (\ref{inequalityge}) (due to the non-triviality assumption). 

On the other hand, any solution to the system of inequalities (\ref{inequality>}) gives us a vector of weights and a threshold. Let ${\bf w}$ be such a solution. Then for any two coalitions $X\in W$, $ Y\in L$ we will have ${\bf v}_{X,Y}\cdot {\bf w}\ge 0$ or 
\[
\sum_{t\in X}w_t  \ge \sum_{s\in Y}w_s. 
\]
Then the smallest sum $\sum_{t\in X}w_t$, where $X\in W$, will still be greater than or equal than the largest sum $\sum_{s\in Y}w_s$, where $Y\in L$. Hence the threshold $q$ can be chosen between them so that
\[
\sum_{t\in X}w_t  \ge q\ge \sum_{s\in Y}w_s. 
\]
The only problem left is that ${\bf w}$ can have negative components and in the definition of a roughly weighted game all weights must be non-negative.  However, due to the monotonicity, if the game $G$ has any rough weights, then it has a non-negative system of rough weights too (with the same threshold). Indeed, if, say weight $w_1$ of the first player is negative, then she cannot be pivotal in any winning coalition. Since her weight is negative, her removal from a winning coalition cannot make it losing.  By monotonicity, deleting her from a losing coalition does not make it winning. In this case the weight $w_1$ can be reset to 0 (or a very small positive weight). We can do this with every negative weight. 
\end{proof}


\section{Trade-robustness and function $f$}

We remind that a sequence of coalitions
\begin{equation}
\label{tradingtransform}
{\cal T}=(\row Xj;\row Yj)
\end{equation}
is a trading transform if the coalitions $\row Xj$ can be converted into the coalitions $\row Yj$ by rearranging players. It can also be expressed as 
\[
|\{i:a\in X_i\}| = |\{i:a\in Y_i\}|\qquad \text{for all $a\in P$}.
\]
It is worthwhile to note that while in (\ref{tradingtransform}) we can consider that no $X_i$ coincides with any of $Y_k$, it is perfectly possible that the sequence $\row Xj$ has some terms equal, the sequence  $\row Yj$ can also contain equal subsets. 
The order of subsets in these sequences is not important, thus in fact we deal with two multisets of coalitions. We will sometimes use the multiset notation and  instead of~(\ref{tradingtransform}) will write
\begin{equation}
\label{tradingtransform_multiset_notation}
{\cal T}=(X_1^{a_1},\ldots,X_k^{a_k}; Y_1^{b_1},\ldots,Y_m^{b_m}),
\end{equation}
where now $\row Xk$ and $\row Yk$ are all distinct,  $\row ak$ and  $\row bm$ are sequences of positive integers such that $\sum_{i=1}^ka_i=\sum_{j=1}^mb_j$ and $Z_i^{c_i}$ denotes $c_i$ copies of $Z_i$ with $Z_i\in \{X_i,Y_i\}$, $c_i\in \{a_i,b_i\}$.\par\medskip

We also have the following obvious algebraic reformulation.

\begin{proposition}
\label{tr=vXY}
Let $\row Xj$ and $\row Yj$ be two sequences of subsets of $[n]$. Then
(\ref{tradingtransform})
is a trading transform iff
\begin{equation}
\label{sumvXY}
{\bf v}_{X_1,Y_1}+\ldots + {\bf v}_{X_j,Y_j}={\bf 0}.
\end{equation}
\end{proposition}

\begin{definition}
A simple game $G$ is called $k$-{\em trade robust} if no trading transform (\ref{tradingtransform}) with $j\le k$ have all coalitions $\row Xj$ winning and all $\row Yj$ losing. $G$ is {\em trade robust} if it is $k$-trade robust for every $k$.
\end{definition}

\begin{proposition}
\label{conditions}
A simple game $G$ is $k$-trade robust if for no $\brow vm\in I(G)$ and for no non-negative integers $\row am$ such that $\sum_{I=1}^m a_i\le k$, we have
\begin{equation}
\label{krobust}
\blcomb avm ={\bf 0}.
\end{equation}
\end{proposition}

\begin{proof} Suppose $G$ is not $k$-trade robust and there exists a trading transform (\ref{tradingtransform}) with $X_i\in W$, $Y_i\in L$ for all $i$ and $j\le k$. Then by Proposition~\ref{tr=vXY} we have
\begin{equation*}
{\bf v}_{X_1,Y_1}+\ldots +{\bf v}_{X_j,Y_j}={\bf 0}
\end{equation*}
with ${\bf v}_i = {\bf v}_{X_i,Y_i}\in I(G)$ so (\ref{krobust}) holds.  On the other hand, if (\ref{krobust}) is satisfied for $\sum_{I=1}^m a_i\le k$, then ${\bf v}_i = {\bf v}_{X_i,Y_i}$ for some $X_i\in W$ and $Y_i\in L$ and  the sequence
\[
(X_1^{a_1},\ldots,X_m^{a_m}; Y_1^{a_1},\ldots,Y_m^{a_m}),
\]
where $X_i^{a_i}$ and $Y_i^{a_i}$ mean $a_i$ copies of $X_i$ and $Y_i$, respectively, is a trading transform  violating $k$-trade robustness.
\end{proof}

Does there exist a positive integer $k$ such that $k$-trade robustness impliy trade robustness? ÊWinder \cite{W} showed that in general no such $k$ exists.  However, if we restrict ourselves with games with $n$ players, then the situation changes and, for each $n$, such a number exists. Of course it will depend on $n$. This is an important result contained in the following theorem. (We note though that the equivalence of (a) and (b) was Êearlier proved by Elgot \cite{Elgot}.)

\begin{theorem}[Taylor-Zwicker, 1992]
\label{Th_TZ92}
The following two conditions are equivalent:
\begin{itemize}
\item $G$ is weighted majority game,
\item $G$ is trade robust,
\item $G$ is $2^{2^n}$-trade robust.
\end{itemize}
\end{theorem}

As this characterisation of weighted games implies, to show that the game $G$ is not weighted majority game, it is sufficient to present a trading transform (\ref{tradingtransform}) where all coalitions $\row Xj$ are winning and all coalitions $\row Yj$ are losing. We will call such a trading transform a {\it certificate of non-weightedness} of $G$.  An interesting question immediately emerges: if we want to check weightedness of a game with $n$ players what is the maximal length of certificates that we have to check?

Let $G=(P,W)$ be a simple game with $|P|=n$.  If $G$ is not weighted we define $f(G)$ to be the smallest positive integer $k$ such that $G$ is not $k$-trade robust.  If $G$ is weighted we set $f(G)=\infty$. The larger the value $f(G)$ the closer is the game $G$ to a weighted majority game. Let us also define
\[
f(n)=\max_G f(G),
\]
where maximum is taken over non roughly-weighted games with $n$ players.  We can also say that  $f(n)$ is the smallest positive integer such that $f(n)$-trade robustness for an $n$-player game implies its weightedness.

\begin{example}
In  Example~\ref{projective} the corresponding $7\times 7$ matrix, composed of vectors ${\bf v}_{X_i,X_i^c}\in I(G)$, $i=1,\ldots,7$, will be:
\begin{equation*}
\left[\begin{array}{rrrrrrrrr}
1 & 1 & 1 & -1 & -1 & -1 & -1 \\
-1 & -1 &  1 & 1 & 1 & -1 & -1 \\
1 & -1 & -1 & -1 & 1 & 1 & -1 \\
1 & -1 & -1 & 1 & -1 & -1 & 1 \\
-1 & 1 & -1 & -1 & 1 & -1 & 1 \\
-1 & -1 & 1 & -1 & -1 & 1 & 1 \\
-1 & 1 & -1 & 1 & -1 & 1 & -1 \\
\end{array}
\right].
\end{equation*}
Its rows sum to the vector $(-1,-1,-1,-1,-1,-1,-1)$. If we also add the vector
$
{\bf v}_{P,\emptyset}=(1,1,1,1,1,1,1)
$
we will get
\[
\sum_{i=1}^7 {\bf v}_{X_i,X_i^c}+{\bf v}_{P,\emptyset}={\bf 0}.
\]
This means that the following eight winning coalitions $(\row X7,P)$, where $P$ is the grand coalition, can be transformed into the following eight losing coalitions:
$(X_1^c,\ldots, X_7^c, \emptyset)$ (note that $\emptyset=P^c$). The sequence
\begin{equation}
\label{Fano_potent_certificate}
(\row X7,P; X_1^c,\ldots, X_7^c,\emptyset)
\end{equation}
is a certificate of non-weightedness of $G$. 
This certificate  is not however the shortest. Indeed, if we take two lines, say, $\{1,2,3\}$ and $\{3,4,5\}$ and swap 2 and 4, then $\{1,3,4\}$ and $\{2,3,5\}$ will not be lines, hence losing coalitions. Thus,  $Fano$ is not 2-trade robust and $f(Fano)=2$.
\end{example}

Theorem~\ref{Th_TZ92} gives us an upper bound for $f(n)$. The following theorem gives a lower bound.

\begin{theorem}[Taylor-Zwicker, 1995]
\label{Th_TZ95}
For each integer $m\ge 2$, there exists a game $Gab_m$, called Gabelman's game, with $(m+1)^2$ players, that is $m$-trade robust but not $(m+1)$-trade robust.
\end{theorem}

Summarising the results of Theorems \ref{Th_TZ92} and \ref{Th_TZ95} in terms of function $f$ we may state

\begin{corollary}
For any $n\ge 2$,
\begin{equation}
\label{theoldbounds}
\lfloor \sqrt{n}\rfloor \le f(n)\le 2^{2^n}. 
\end{equation}
\end{corollary}

 As Taylor and Zwicker noted in \cite{TZ} for most non-weighted games the value of $f(n)$ is $2$. The closer the game to a weighted majority game the longer is the certificate and it is harder to find it. 

\section{A new upper bound for $f$}

In what follows  we use the following notation. Let ${\bf x}\in \R^n$. Then we write ${\bf x}\gg {\bf 0}$ iff $x_i> 0$ for all $i=1,\ldots,n$. We also write  ${\bf x}>{\bf 0}$ iff $x_i\ge 0$ for all $i=1,\ldots,n$ with this inequality being strict for at least one $i$, and ${\bf x}\ge {\bf 0}$ iff $x_i\ge 0$ for all $i=1,\ldots,n$. In this section we will need the following result which may be considered as a folklore. 

\begin{theorem}
\label{ramtheorem}
Let $A$ be an $m\times n$ matrix with rational coefficients with rows ${\bf a}_i\in\Q^n$, $i=1,\ldots, m$. Then the system of linear inequalities $A{\bf x}\gg {\bf 0}$, ${\bf x}\in\R^n$, has no solution  iff there exist non-negative  integers $\row rm$, of which at least one  is positive, such that
\begin{equation}
\label{ram}
\blcomb ram ={\bf 0}.
\end{equation}
\end{theorem}

A proof can be found in \cite{TZ}, Theorem~2.6.4, p. 71 or in \cite{Mu}, Lemma~7.2.1, p. 192.

\begin{theorem}
\label{log}
The following statements for a simple game $G$ with $n$ players are equivalent:
\begin{enumerate}
\item[(a)] $G$ is weighted,
\item[(b)] $G$ is $N$-trade robust for $N=(n+1)2^{\frac{1}{2}n\log_2n}$.
\end{enumerate}
\end{theorem}

\begin{proof}
We only need  to prove that (b) implies (a). Suppose $G$ is not weighted. Then by Proposition~\ref{w&rw} the system of inequalities
\[
{\bf v}\cdot {\bf x}>0,\qquad {\bf v}\in I(G),
\]
is inconsistent. By Theorem~\ref{ramtheorem} there exist vectors $\brow vm\in I(G)$ and non-negative integers $\row rm$ such that $\blcomb rvm ={\bf 0}$. Let $m$ be minimal with this property. Then all $r_i$'s are non-zero, hence positive. By a standard linear algebra argument (see, e.g. Theorem 2.11 from \cite{Gale}) we may then assume that $m\le n+1$ and that the system of vectors $\{\brow v{m-1}\}$ is linearly independent. We will assume that $m=n+1$ as it is the worst case scenario. Let $A=({\bf a}_1\, {\bf a}_2\,\ldots\, {\bf a}_{n+1})$ be the $(n+1){\times} n$ matrix, which $i$th row is ${\bf a}_i={\bf v}_i$ for $i=1,2,\ldots, n+1$. The null-space of the matrix $A$ is one-dimensional, and since $(\brow r{n+1})$ is in it, then the coordinates in any solution are either all positive or all negative. Looking for a solution of the system $x_1{\bf v}_1+\ldots +x_n{\bf v}_n=-{\bf v}_{n+1}$, by Cramer's rule we find $x_i=\det A_i/\det A$, where $A=({\bf a}_1\,\ldots\, {\bf a}_n)$ and $A_i$ is obtained when ${\bf a}_i$ in $A$ is replaced with ${\bf a}_{n+1}$. Thus
\begin{equation}
\det A_1\,{\bf a}_1+\ldots + \det A_n\,{\bf a}_n+\det A\,{\bf a}_{n+1}={\bf 0}
\end{equation}
and by Hadamard's inequality \cite{Hd} we have $\det A_i\le n^{n/2}=2^{\frac{1}{2}n\log_2n}$. The sum of all coefficients is smaller than or equal to $(n+1)2^{\frac{1}{2}n\log_2n}=N$. Since $G$ is $N$-trade robust this is impossible by Proposition~\ref{tr=vXY}.
\end{proof}

\begin{corollary}
$f(n)\le (n+1)2^{\frac{1}{2}n\log_2n}$.
\end{corollary}


\section{A new lower bound for $f(n)$}

Let ${\bf w}=(\row wn)$ be a vector with non-negative coordinates. There may be some linear relations between the coordinates of ${\bf w}$. Let us define those relations that will be important for us. Let $X,Y$ be subsets of $[n]$ such that $X_i\cap Y_i=\emptyset$.  If ${\bf v}_{X,Y}\cdot {\bf w}=0$, which is the same as $\sum_{i\in X}w_i - \sum_{j\in Y}w_j=0$, then we say that the coordinates of ${\bf w}$ are in the relation which corresponds to the vector ${\bf v}= {\bf v}_{X,Y}\in T^n$. \par\medskip

Given $X\subseteq [n]$ we may then introduce $w(X)=\sum_{i\in X}w_i$. For two subsets $X,Y\subseteq [n]$ we write $X\sim Y$ if $w(X)=w(Y)$.  Of course, if this happens, then the coordinates of ${\bf w}$ satisfy the equation ${\bf v}_{X,Y}\cdot {\bf w}=0$. Suppose $X\sim Y$, then the equivalence $X'\sim Y'$, where $X'=X\setminus (X\cap Y)$ and $Y'=Y\setminus (X\cap Y)$ will be called {\em primitive} and $X\sim Y$ will be called a consequence of $X'\sim Y'$.

\begin{example}
\label{f(9)ge 4}
Consider the vector of weights ${\bf w}=(w_1,w_2,w_3,w_4,w_5)=(1,2,5,6,10)$. Then 
\begin{equation}
\label{relations}
w_1+w_3=w_4,\quad w_1+w_4=w_2+w_3,\quad w_2+w_5=w_1+w_3+w_4,\quad w_3+w_4=w_1+w_5
\end{equation}
are relations which correspond to vectors
\begin{equation}
\label{vectors}
(1,0,1,-1,0),\quad (1,-1,-1,1,0),\quad (-1,1,-1,-1,1),\quad (-1,0,1,1,-1),
\end{equation}
respectively. 
It is easy to check that there are no other relations between the coordinates of ${\bf w}$. (Note that we do view $w_1+w_2+w_3=w_2+w_4$ and  $w_1+w_3=w_4$ as the same relation.)

We have\footnote{Here and below we omit curly brackets in the set notation}
\begin{center}
\vspace{4mm}
\begin{tabular}{|c|c|}
\hline
Primitive equivalence & Total weight of equal subsets\\
\hline
$13\sim 4$   & 6 \\
$14\sim 23$ & 7 \\
$\ 25\sim 134$ & 12 \\
$34\sim 15$ & 11\\ 
\hline
\end{tabular}
\vspace{4mm}
\end{center}
Apart from these four equivalences and their consequences there are no other equivalences.
\end{example}

\begin{definition}
Let ${\bf w}=(\row wn)$ be a vector of non-negative coordinates. We will say that ${\bf w}$ satisfies $k$th {\em Fishburn's condition} if  there exist distinct vectors ${\bf v}_i={\bf v}_{X_i,Y_i}\in T^n$, $i=1,\ldots, k$, with $X_i\cap Y_i=\emptyset$, such that:
\begin{itemize}
\item $X_i\sim Y_i$ for $i=1,\ldots,k$, that is ${\bf v}_i\cdot {\bf w}=0$ is a relation for the  coordinates of ${\bf w}$.
\item Apart from $\brow vk$ there are no other relations.
\item $\sum_{i=1}^{k}{\bf v}_i={\bf 0}$,
\item No proper subset of vectors of the system $\{\brow v{k}\}$ is linearly dependent.
\end{itemize}
\end{definition}

For us the importance of this condition is shown in the following

\begin{theorem}
\label{doubling}
Let ${\bf w}=(\row wn)$, $n>2$, be a vector with positive coordinates which satisfies the $k$th Fishburn condition. Then there exists a simple game on $n+k$ players which is $(k-1)$-trade robust but not $k$-trade robust.
\end{theorem}

\begin{proof}
Suppose vectors ${\bf v}_i={\bf v}_{X_i,Y_i}\in T^n$, $i=1,\ldots, k$ are those that are required for the $k$th Fishburn condition. Then by Proposition~\ref{tr=vXY} the sequence ${\cal T}=(\row Xk;\row Yk)$ is a trading transform. Let $w(X)$ be the total weight of the coalition $X$.  Then we have $s_i=w(X_i)=w(Y_i)$ for $i=1,\ldots,k$. Let $N$ be any positive integer greater than $2w(P)$. 

We define 
\[
P'=P\cup \{n+1,\ldots, n+k\},\quad X_i'=X_i\cup \{n+i\},\quad Y_i'=Y_i\cup \{n+i\}. 
\]
Then ${\cal T}_1=(X_1'\ldots, X_k'; Y_1'\ldots, Y_k')$ is obviously also a trading transform. Let us give weight $N-s_i$ to $n+i$. We will call these new elements {\em heavy}. Then
\[
w(X_1')=\ldots=w(X_k')=w(Y_1')=\ldots=w(Y_k')=N.
\]
Moreover, we are going to show that no other subset of $P'$ has weight $N$. Suppose there is a subset $Z\subset P'$ whose total weight is $N$ and which is different from any of the $X_1',\ldots X_k'$ and $Y_1',\ldots Y_k'$. Since $N>2w(P)$ and   $2N-s_i-s_j\ge 2N-2w(P)>N$ holds for any $i,j \in \{1, \ldots, k\}$, then $Z$ must contains no more than one heavy element, say $Z$ contains $n+i$. Then for $Z'=Z\setminus \{n+i\}$ we have $w(Z')=N-(N-s_i)=s_i$ which implies $Z'=X_i$ or $Z'=Y_i$, a contradiction.

Let us now consider the game $G$ on $[n]$ with rough voting representation $[N;\row wn]$, where $X_1',\ldots X_k'$ are winning and $Y_1',\ldots Y_k'$ are losing. Since these are the only subsets on the threshold, the game is fully defined. ${\cal T}_1$ becomes a certificate of non-weightedness for $G$ so it is not $k$-trade robust. Let us prove that it is $(k-1)$-trade robust. Suppose on the contrary, there exists a certificate of non-weightedness for $G$
\begin{equation}
\label{trading_transform_U_V}
{\cal T}_2=(\row Us;\row Vs),\qquad s\le k-1,
\end{equation}
where $\row Us$ are all winning and $\row Vs$ are all losing. Then this can happen only if all these vectors are on the threshold, that is,
\[
w(U_1)=\ldots=w(U_s)=w(V_1)=\ldots=w(V_s)=N,
\]
hence $U_i\in \{X_1',\ldots X_k'\}$ and $V_j\in \{Y_1',\ldots Y_k'\}$. As was proved, any of $U_i$ and any of $V_j$ contain exactly one heavy player. Suppose, without loss of generality that $U_1=X_{i_1}'=X_{i_1}\cup \{n+i_1\}$. Then we must have at least one player $n+i_1$ among the $\row Vs$. Without loss of generality we may assume that $V_1=Y_{i_1}'=Y_{i_1}\cup \{n+i_1\}$. We may now cancel $n+i_1$ from the trading transform (\ref{trading_transform_U_V}) obtaining a certificate of non-weightedness
\[
{\cal T}_3=(X_{i_1},U_2,\ldots,U_s;Y_{i_1}, V_2,\ldots, V_s).
\]
Continuing this way we will come to a certificate of non-weightedness
\[
{\cal T}_4=(X_{i_1},\ldots, X_{i_s};Y_{i_1},\ldots, Y_{i_s}),
\]
which by Proposition 3 will give us ${\bf v}_1+\ldots+{\bf v}_s={\bf 0}$. The latter contradicts to the fact that no proper subset of vectors of the system $\{\brow v{k}\}$ is linearly dependent.
\end{proof}

Fishburn proved the following combinatorial lemma which plays the key role in our construction of games.

\begin{lemma}[Fishburn, 1997]
\label{Fishburn_example}
For every $n\ge 5$ there exists a vector of weights ${\bf w}=(\row wn)$ which satisfies the $(n-1)$-th Fishburn condition.
\end{lemma}

\begin{proof}
See  \cite{PF1,PF2}.
\end{proof}

\begin{corollary}
\label{ourtheorem}
For each integer $n\ge 5$, there exists a game with $2n-1$ players, that is $(n-2)$-trade robust but not $(n-1)$-trade robust. Moreover, $n-1\le f(2n-1)$. For an arbitrary $n$
\begin{equation}
\label{in_combined}
\left\lfloor \frac{n-1}{2}\right\rfloor \le f(n).
\end{equation}
\end{corollary}

\begin{proof}
The first part follows immediately from Theorem~\ref{doubling}  and Lemma~\ref{Fishburn_example}. Indeed, in this case the length of the shortest certificate of non-weightedness is $n-1$. We also trivially have $n-1\le f(2n)$.  These two inequalities can be combined into one inequality (\ref{in_combined}).
\end{proof}

\begin{example}[Continuation of Example~6]
Suppose $P=[9]$. The first five players get weights $(w_1,w_2,w_3,w_4,w_5)=(1,2,5,6,10)$. The other four players get weights $(w_6,w_7,w_8,w_9)=(106,105,100,101).$ Then we get the following equivalences:
\begin{center}
\begin{tabular}{|c|c|}
\hline
Equivalence & Total weight of subsets\\
\hline
$13{6}\sim 4{6}$   & \ 6+{106}={ 112} \\
$\ 147\sim 237$ & \ 7+{105}={ 112} \\
$\ \ 258\sim 1348$ & 12+{100}={ 112} \\
$\ 349\sim 159$ & 11+{101}={ 112}\\ 
\hline
\end{tabular}
\end{center}
We define 
\begin{itemize}
\item Coalitions whose total weight  is $>{ 112}$ are winning.
\item Coalitions whose total weight  is $<{ 112}$ are losing.
\item $46,  237, 1348, 159$ are winning.
\item $136, 147, 258, 349$ are losing.
\end{itemize}
This gives us a game with a shortest certificate of length 4, that is, $f(9)\ge 4$. Gabelman's example gives $f(9)\ge 3$.
\end{example}

Fishburn \cite{PF1,PF2} conjectured that a system of $n$ weights cannot satisfy a $n'$th Fishburn condition for $n'\ge n$. This appeared to be not the case. Conder and Slinko \cite{CS} showed that a system of 7 weights can satisfy 7th Fishburn condition.  Conder\footnote{Reported in \cite{SM}} checked that this is also the case for $7\le n\le 13$. Marshall \cite{SM} introduced a class of optimus primes and showed that if $p$ is such a  prime then a system of $p$ weights satisfying $p$th Fishburn condition exists. Although computations show that optimus primes are quite numerous \cite{SM}, it is not known if there are infinitely many of them. The definition of an optimus prime is too technical to give it here.

\begin{corollary}
\label{ourtheorem2}
For each integer $7\le n\le 13$ and also for any $n$ which is an optimus prime,  there exists a game with $2n$ players, that is $(n-1)$-trade robust but not $n$-trade robust. Moreover, $n\le f(2n)$ for such $n$.
\end{corollary}

\begin{proof}
Follows from Theorem~\ref{doubling} along the lines of Corollary~\ref{ourtheorem}.
\end{proof}


\section{A criterion  for  rough weightedness}

The following result that we need in this section is not new either. Kraft et al \cite{KPS} outlined the idea of its proof without much details. Since this result is of fundamental importance to us, we give a full proof in the appendix.

\begin{theorem}
\label{ram1theorem}
Let $A$ be an $m\times n$ matrix with rational coefficients.
Let ${\bf a}_i\in\Q^n$, $i=1,\ldots, m$ be the rows of $A$.
Then the system of linear inequalities $A{\bf x}\ge {\bf 0}$ has no  non-negative solution ${\bf x}\ge {\bf 0}$, other than ${\bf x}={\bf 0}$, iff there exist non-negative integers $\row rm$ and a vector ${\bf u}$ whose all entries are positive integers such that
\begin{equation}
\label{ram1}
\blcomb ram +{\bf u}={\bf 0}.
\end{equation}
\end{theorem}

\begin{definition}
A certificate of non-weightedness, which includes $P$ and $\emptyset $ we will call {\em potent}. 
\end{definition}

We saw such a certificate in (\ref{Fano_potent_certificate}) for Fano plane game. Now we can give a criterion for a game to be roughly weighted.

\begin{theorem}[Criterion of rough weightedness]
\label{Cr_0f_nrw}

The game $G$ with $n$ players is roughly weighted  if one of the two equivalent statements hold:
\begin{enumerate}
\item[(a)] for no positive integer  $j\le (n+1)2^{\frac{1}{2}n\log_2n}$ there exist a potent certificate of non-weightedness of length $j$,
\item[(b)] for no positive integer  $j\le (n+1)2^{\frac{1}{2}n\log_2n}$ there exist $j$ vectors $\brow vj\in I(G)$  such that
\begin{equation}
\label{vjs}
{\bf v}_1+\ldots+{\bf v}_j+{\bf 1}={\bf 0},
\end{equation}
where ${\bf 1}=(1,1,\ldots,1)$.
\end{enumerate}
\end{theorem}

\begin{proof}
Since ${\bf 1}={\bf v}_{P,\emptyset}$,  by Proposition~\ref{tr=vXY} we know that (a) and (b) are equivalent. We also note that, as in Theorem~\ref{log}, it can be shown that if a relation (\ref{vjs}) holds in an $n$-player game for some $j$, then there is another such relation with $j\le (n+1)2^{\frac{1}{2}n\log_2n}$. \par

Suppose that (\ref{vjs}) is satisfied but $G$ is roughly weighted. By Proposition~\ref{w&rw} this means that the system
\[
{\bf v}_i\cdot {\bf x}\ge 0,\qquad i=1,2,\ldots, j
\]
has a non-zero non-negative solution, let us call it also ${\bf x}_0$. Then
\[
0=({\bf v}_1+\ldots+{\bf v}_j+{\bf 1})\cdot {\bf x}_0\ge |{\bf x}_0|> {\bf 0},
\]
where $|{\bf x}|$ denotes the sum of all coordinates of ${\bf x}$. This is a contradiction.

Let us suppose now that a system of rough weights for the game $G$ does not exist. Then the system (\ref{inequalityge}) has no solution and by Theorem~\ref{ram1theorem} there exist vectors $\brow vm\in I(G)$ and a vector ${\bf u}$ whose all coordinates are positive integers and such that
\begin{equation}
\label{vvu}
{\bf v}_1+\ldots+{\bf v}_m +{\bf u} ={\bf 0}
\end{equation}
(where not all of the vectors ${\bf v}_i$ may be different). Let us consider the relation (\ref{vvu})  with the smallest sum $|{\bf u}|=u_1+\ldots+u_n$ of coordinates of ${\bf u}$. If $|{\bf u}|=n$ we are done. Suppose $u_i>1$ for some $i\in [n]$. Then we can find $j\in [n]$ such that the $i$th coordinate of ${\bf v}_j$ is $-1$. Then ${\bf v}'_j={\bf v}_j+{\bf e}_i\in I(G)$ and we can write
\[
{\bf v}_1+\ldots+{\bf v}_j'+\ldots+{\bf v}_m +{\bf u}' ={\bf 0},
\]
where ${\bf u}'={\bf u}-{\bf e}_i$. Since all coordinates of ${\bf u}'$ are positive integers and their sum is $|{\bf u}|-1$ this contradicts to the minimality of $|{\bf u}|$.
\end{proof}

The game Fano can be generalised in several different ways. We will consider two such generalisations.

\begin{example}[Hadamard games]
An Hadamard matrix $H$ of order $n\times n$ is a matrix with entries $\pm 1$ such that $H^TH=HH^T=I_n$, where $I_n$ is the identity matrix of order $n$.  The latter condition is equivalent to the system of rows of $H$ as well as the system of columns being orthogonal. The standard example of Hadamard matrices is the sequence 
\begin{equation}
\label{Had}
H_1=\twomat 111{-1},\qquad H_{k+1}=\twomat {H_k}{H_k}{H_k}{{-H_k}}
\end{equation}
discovered by Sylvester in 1867 \cite{Syl}. Here $H_k$ is $2^k\times 2^k$ matrix. It is known that the order of an Hadamard matrix must be divisible by four and the hypothesis is being tested that for any $k$ an Hadamard matrix of order $4k$ exists. However it has not been proven and the smallest $k$ number for which it is not known whether or not an Hadamard matrix of order $4k$ exists is currently  167 \cite{428}.

Suppose now that an Hadamard matrtix of order $n>4$ exists. In a usual way (by multiplying certain rows and columns by $-1$, if necessary) we may assume that all integers in the first row and in the first column of $H$ are $1$. Then we consider the matrix $\overline{H}$ which is $H$ without its first row and its first column.  The game $HG_{n-1}=(P,W)$ will be defined on the set of players $P=[n-1]$. We consider the rows of $\overline{H}$ and view them as the characteristic vectors of subsets $\row X{n-1}$. Any two rows of $H$ are orthogonal which implies that the number of places where these two rows differ are equal to the number of places where they coincide. However, if $X_i\cap X_j= \emptyset $, then the number of places where the two rows differ would be $2(n/2-1)=n-2$ which is greater than $n/2$ for $n>4$. Hence $X_i\cap X_j\ne \emptyset $ for any $i,j$.  Let us consider $\row X{n-1}$ as minimal winning coalitions of $HG_{n-1}$.  It is easy to see that the Hadamard game $HG_7$ obtained from $H_3$ is the Fano plane game.
\end{example}

\begin{definition}
A game with $n$ players will be called {\em cyclic} if the charactiristic vectors of minimal winning coalitions consist of a vector  ${\bf w}\in\Z_2^n$ and all its cyclic permutations. We will denote it $C({\bf w})$.
\end{definition}

It is not difficult to see that  the game Fano in Example~\ref{projective} is cyclic.

\begin{theorem}\label{pr_cyc}
Suppose that the Hamming weight of ${\bf w}\in\Z_2^n$ is smaller than $n/2$.  Suppose the game $C({\bf w})$ is proper. Then it is not roughly weighted.
\end{theorem}

\begin{proof}
Let $X_1,\ldots, X_n$ correspond to the characteristic vectors, which are ${\bf w}$ and all its cyclic permutations. Suppose that the Hamming weight of ${\bf w}$ is $k$.  Then the sequence 
\[
{\cal T}=(X_1,\ldots,X_n,\underbrace{P,\ldots, P}_{n-2k};X_1^c,\ldots,X_n^c,\underbrace{\emptyset,\ldots, \emptyset}_{n-2k})
\] 
is a trading transform. Since the game is proper, $X_1,\ldots, X_n \in W$ and $X_1^c, \ldots, X_n^c \in L.$ Thus by Theorem~\ref{Cr_0f_nrw} the game $C({\bf w})$ is not roughly weighted.
\end{proof}

Richardson  \cite{Rich}  studied the following class of games that generalise the Fano game.
Let $q=p^r$, where $p$ is prime. Let $GF(q)$ be the Galois field with $q$ elements and $PG(n,q)$ be the projective $n$-dimensional space over $GF(q)$. It is known \cite{Rich, Hall} that $PG(n,q)$ contains $\frac{q^{n+1}-1}{q-1}$ points and any its $(n-1)$-dimensional subspace consists of $\frac{q^n-1}{q-1}$ points. Any two such subspaces have $\frac{q^{n-1}-1}{q-1}$ points in their intersection. 

We define a game  $Pr_{n,q}=(PG(n,q),W)$ by defining the set $W^m$ of all minimal winning coalition be the set of all $(n-1)$-dimensional subspaces of $PG(n,q)$.
This class of games is known as {\em projective games}. These games are cyclic by Singer's theorem (see, e.g., \cite{Hall}, p. 156).

\begin{corollary}
Any projective game is not roughly weighted.
\end{corollary}

\begin{proof}
We note that, since any two winning coalitions of $Pr_{n,q}$ intersect, this game is proper.
By Singer's Theorem  any projective game $Pr_{n,q}$ is cyclic. 
Now the statement follows from Theorem~\ref{pr_cyc}.
\end{proof}



\section{AT-LEAST-HALF Property}

We can also characterise rough weightedness in terms of EL sequences similar to Theorem~2.4.6 of \cite{TZ}. We remind to the reader that a coalition is {\em blocking} if it is a complement of a losing coalition. A sequence of coalitions $(\row Z{2k})$ is called an EL {\em sequence} of degree $k$ (see \cite{TZ}, p. 61) if half of its coalitions are winning and half are blocking.  

\begin{definition}
A simple game satisfies AT-LEAST-HALF  PROPERTY of degree $k$ if any EL sequence of degree $k$ or less has some player occurring in at least half of the coalitions in the sequence.
\end{definition}

\begin{theorem}
For a simple game $G$ the following are equivalent:
\begin{enumerate}
\item[(i)] $G$ is roughly weighted.
\item[(ii)] $G$ has at-least-half property of degree $(n+1)2^{\frac{1}{2}n\log_2n}$.
\end{enumerate}
\end{theorem}

\begin{proof}
Suppose $G$ is roughly weighted and let $[q;\row wn]$ be its rough voting representation. Let $Z=(\row Z{2k})$ be an EL sequence. Then any winning coalition in $Z$ has weight of at least $q$ and any blocking coalition in $Z$ has weight of at least $\Sigma -q$, where $\Sigma=\sum_{i=1}^nw_i$. The total weight of coalitions in $Z$ is therefore at least $kq+k(\Sigma -q)=k\Sigma$. If (ii) is not satisfied, then any player occurs in the sequence less than $k$ times and the total weight of coalitions in $Z$ is therefore strictly less than $\sum_{i=1}^nkw_i=k\Sigma$, which is a contradiction. Hence (i) implies (ii).

Suppose now that (ii) is satisfied but $G$ is not roughly weighted. Then there exist a potent certificate of non-weightedness
\[
{\cal T}=(\row Xk, P;\row Yk,\emptyset),
\]
where $k\le (n+1)2^{\frac{1}{2}n\log_2n}$. Then the sequence $Z=(\row Xk;Y_1^c, ..., Y_k^c)$ is an EL sequence. Consider an arbitrary player $a$. For a certain positive integer $s$ it occurs $s$ times in the subsequence $Z'=(\row Xk, P)$ and $s$ times in the subsequence $Z''=(\row Yk, \emptyset)$. Thus in $Z$ it will occur $(s-1)+(k-s)=k-1$ times, which is less than half of $2k$ and at-least-half property does not hold. By the Criterion of rough weightedness (Theorem~\ref{Cr_0f_nrw}) we conclude that $G$ is roughly weighted.
\end{proof}


\section{Function $g$.}

Suppose now that we have to check if a game $G$ is roughly weighted or not. According to the Criterion of Rough Weightedness (Theorem\ref{Cr_0f_nrw}) we have to check if there are any potent certificates of non-weightedness. We have to know where to stop while checking those. We will define a new function for this.  If the game is roughly weighted let us set $g(G)=\infty$.  Alternatively, $g(G)$ is the  length of the shortest potent certificate of non-weightedness for $G$. We also define a function
\[
g(n)=\max_G g(G),
\]
where maximum is taken over non roughly-weighted games with $n$ players. Checking rough weightedness we then have to check all potent certificates of non-weightedness up to a length $g(n)$.\par\medskip

For the Fano plane game in Example~\ref{projective} we have a potent certificate of non-weightedness (\ref{Fano_potent_certificate}) which has length $8$. We will prove that this is the shortest potent certificate for this game.

\begin{theorem}
\label{g(projective)}
$g(Fano)=8$. 
\end{theorem}

\begin{proof}
We claim that any ${\bf v}\in I(G)$ has the sum of coefficients $|{\bf v}|=v_1+\ldots+v_n\ge -1$. Indeed, such a vector would be of the form ${\bf v}={\bf v}_{X,Y}$, where $X$ is winning and $Y$ is losing. Since $X$ is winning ${\bf v}$ has at least three positive ones and since $Y$ is losing it has at most four negative ones (as all coalitions of size five are winning).

Suppose now there is a sum
\[
{\bf v}_1+\ldots +{\bf v}_j+{\bf 1}={\bf 0},
\]
where ${\bf v}_i\in I(G)$, which represents a potent certificate of non-weightedness of length less than eight. In this case  $j\le 6$. By the observation above the sum of coefficients of vectors $\brow v6$ is at least $-6$. Since the sum of coefficients of ${\bf 1}$ is seven, we obtain a contradiction.
\end{proof}

\begin{theorem}
$f(HG_n)=2$ and $g(HG_n)=n+1$ for all $n$.
\end{theorem}

\begin{proof}
Repeats the respective proofs of for Fano.
\end{proof}

Let us now deal with the lower and upper bounds for $g$.

\begin{theorem}
\label{th_n_2n+3}
For any $n \geqslant 5$
$$2n+3 \leqslant g(n) \leqslant (n+1)2^{\frac{1}{2} n \log_2 n}. $$
\end{theorem}

\begin{proof}
Due to Theorem~\ref{Cr_0f_nrw} we need only to take care of the lower bound. For this we need to construct a game $G$ with $n$ players such that $g(G)=2n+3$.

Let us define the game $G_{n,2} = ([n], W)$ where 
\begin{itemize}
\item $\{ 1, 2 \} \in W$ and $ \{ 3, 4, 5 \} \in W;$
\item if $|S|>3$ then $S \in W.$
\end{itemize}

Note that all losing coalition have cardinality at most three.

We note that the trading transform
\begin{align*}
{\cal T}=&\{\{1, 2 \}^n, \{ 3, 4 , 5 \}^{n+2}, P; \{2, 3, 5 \}^{3}, \{2, 3, 4 \}^{3},\\
&\underbrace{\{2, 3, 6 \}, \ldots ,\{2, 3, n \}}_{n-5},\{1, 3, 4 \}, \{1, 3, 5\},\{1, 4, 5\}^{n-1},\emptyset\}
\end{align*}
is a potent certificate of non-weightedness for $G$. Its length $2n+3$  is minimal. To prove this we will use the idea introduced in the proof of Theorem~\ref{g(projective)}. Since all losing coalition have cardinality at most three, any ${\bf v}\in I(G)$ has the sum of coordinates $v_1+\ldots+v_n\ge -1$ and that all such vectors ${\bf v}$ with $v_1+\ldots+v_n= -1$ have the form ${\bf v}_{\{1,2\},Y}$ for $Y$ being a losing $3$-player coalition.

Suppose now there is a sum
\[
{\bf v}_1+\ldots +{\bf v}_k+{\bf 1}={\bf 0},
\]
where ${\bf v}_i\in I(G)$, which represents a potent certificate
$
{\cal T}= ( X_{1}, \ldots , X_{k}, P; Y_{1}, \ldots, Y_{k},\emptyset).
$
Due to the comment above, at least $n$ vectors among $\brow vk$ must have the sum of coordinates $-1$ and hence be of the form ${\bf v}_{\{1,2\},Y}$, where $Y$ is a losing 3-player coalition. This means that there are at least $n$ sets $\{1,2\}$ among $\row Xk$. Add the grand coalition and we obtain that the union $X_{1}\cup  \ldots \cup X_{k}\cup P$ has at least $n+1$ elements 1 and at least $n+1$ elements 2. At the same time no losing coalition can contain both 1 and 2. Hence we will need at least $2n+2$ losing coalitions $\row Yk$ to achieve the equality $X_{1}\cup  \ldots \cup X_{k}\cup P=Y_{1}\cup  \ldots \cup Y_{k}\cup \emptyset$. Hence ${\cal T}$ is minimal.
\end{proof}



\section{Further properties of functions $f$ and $g$.}

What can we say about the relation between $f(n)$ and $g(n)$? One thing that can be easily observed is given in the following theorem.

\begin{theorem}
$f(n)\le g(n)-1$.
\end{theorem}

\begin{proof}
Suppose $g(n)$ is finite and there is a sum
\[
{\bf v}_1+\ldots +{\bf v}_m+{\bf 1}={\bf 0}, \qquad m=g(n)-1,
\]
where ${\bf v}_i\in I(G)$, which represents a potent certificate of non-weightedness of length $g(n)$. We will show that $\brow vm$ can absorb ${\bf 1}={\bf e}_1+\ldots+{\bf e}_n$ and remain in $I(G)$. Let us start with ${\bf e}_1$. One of the vectors, say ${\bf v}_i$, will have $-1$ in the first  position. Then we replace ${\bf v}_i$ with ${\bf v}_i+{\bf e}_1$. The new vector is again in $I(G)$. It is clear that we can continue absorbing ${\bf e}_i$'s until all are absorbed.
\end{proof}

Let us talk about duality in games. 
%
The dual game of a game $G=(P,W)$ is defined to  be $G^{*} = (P,L^{c}).$ This is to say that in the game $G^{*}$ dual to a game $G$ the winning coalitions are exactly the complements of losing coalitions of $G$.

Shapley \cite{Sh} proved that
for any simple game $G$:
\begin{enumerate}
\item[(a)] $G=G^{**}$.
\item[(b)] $G^{*}$ is proper if and only if $G$ is strong.
\item[(c)] $G^{*}$ is strong if and only if $G$ is proper.
\end{enumerate}
The operation of taking the dual is known to  preserve both weightedness and rough weightedness \cite{TZ}. 
%
%
These well-known facts we will sometimes use without quoting them formally.

\begin{theorem}
Let g be a simple game, then $f(G)=f(G^*)$ and $g(G) = g(G^*).$
\end{theorem}
\begin{proof}
Firstly, we shall prove the statement about $f$. Let $G=(P,W)$ be a simple game and ${\cal T} = (X_1, \ldots, X_k;Y_1,\ldots,Y_k)$ be a certificate of non-weightedness of $G,$ then a sequence of even length ${\cal T}^* = (Y_1^c,\ldots,Y_k^c;X_1^c, \ldots, X_k^c)$ will be a trading transform for $G^*$. Indeed, it is not difficult to see that  $X_1^c, \ldots, X_k^c$ are loosing coalitions in $G^{*}$ and $Y_1^c,\ldots,Y_k^c$ are winning. Hence $f(G)\le f(G^*)$. However, due to Theorem 1 (a) we have $f(G^*)\le f(G^{**})=f(G)$.
Proof of the second part of the theorem is  similar.
\end{proof}


\section{Games with a small number of players}

\begin{definition}
We say that a player in a game is a  {\em dictator} if a coalition is winning if and only if this player belongs to it. If all coalitions containing a particular player are winning (but there may be other winning coalitions), this player is called a {\em weak dictator}. A player will be called a {\em vetoer} if she is contained in the intersection of all winning coalitions.
\end{definition}

\begin{proposition}\label{Dictator}
Suppose $G$ is a simple game with $n$ players. Then $G$ is roughly weighted if any one of the following three conditions holds:
\begin{enumerate}
\item[(a)] $G$ has a weak dictator.
\item[(b)] $G$ has a vetoer.
\item[(c)] $G$ has a  losing coalition that consists of $n-1$ players.
\end{enumerate}
\end{proposition}

\begin{proof}
To prove (a) we simply give weight 1 to the weak dictator and 0 to everybody else. The rough quota must be set to 0 (this is possible since we have a non-zero weight). To prove (b) suppose that $v$ is a vetoer for this game and that there exists a potent certificate of non-representability 
\begin{equation}
\label{certificate_of_non-weightedness}
{\cal T}=(\row Xk,P;\row Yk,\emptyset),
\end{equation}
  Then $v$ belongs to all winning coalitions $\row Xk, P$ of this trading transform but it cannot belong to all losing coalitions since $\emptyset$ does not contain $v$. This contradiction proves (b). Part (c) follows from (b) since if $Y$ is a losing coalition of size $n-1$, then the player $v$ such that $\{v\}=P\setminus Y$ is a vetoer.
\end{proof}

\begin{theorem}\label{2_proper}
Let $G$ be a simple proper game with $n$ players. If $G$ has a two-player winning coalition, then $G$ is a roughly weighted game.
\end{theorem}

\begin{proof}
Suppose $G$ has a two-player winning coalition. Players from this coalition we will call heavy. Since $G$ is proper  every winning coalition $X$ must contain at least one heavy player and any losing coalition $Y$ can contain at most one. Then a potent certificate of non-weightedness  (\ref{certificate_of_non-weightedness}) cannot exist since all coalitions $\row Xk,P$ will contain at least $k+2$ heavy players while coalitions $\row Yk,\emptyset $ will contain at most $k$.  Therefore $G$ is roughly weighted.
\end{proof}

\begin{corollary}
\label{n-2_strong}
Let $G$ be a simple strong game with n players. If $G$ has a  losing coalition of cardinality $n-2$, then $G$ is a roughly weighted game.
\end{corollary}

\begin{proof}
The dual game $G^*$ is proper and it will have a two-player winning coalition, hence
Theorem~\ref{2_proper} implies this corollary since the operation of taking a dual game preserves rough weightedness. 
\end{proof}

\begin{proposition}\label{4players}
Every game with $n\leqslant 4$ players is a roughly weighted game.
\end{proposition}

\begin{proof}
It is obvious for $n=1,2,3$ as all games in this case are known to be weighted.  Let $G$ be a 4-player simple game. If $G$ has a one-player winning coalition it is roughly weighted by Proposition~\ref{Dictator}(a). Also, if $G$ has a 3-player losing coalition, it is roughly weighted by Proposition~\ref{Dictator}(c). Thus we may assume that all winning coalitions have at least two players and all losing coalitions have at most two players.

Suppose, that there is a potent certificate of non-weightedness of the form (\ref{certificate_of_non-weightedness}).
As we mentioned, we may assume that $ |X_{i}| \ge 2$ and $|Y_{i}|\le 2$ for every $i=1,\ldots,k$. But in this case the multisets $X_{1}\cup \ldots \cup X_{k}\cup P$ and $Y_{1}\cup \ldots\cup Y_{k}\cup \emptyset$ have different cardinalities and cannot be equal which is a contradiction.
\end{proof}

\begin{theorem}
\label{rough_weights_for 5}
Any simple game $G$ with five players, which is either proper or strong,  is a roughly weighted game.
\end {theorem}

\begin{proof}
Assume on the contrary that there exists a game $G=([5],W)$, which is  not roughly weighted.  It means that there is a potent certificate of non-weightedness (\ref{certificate_of_non-weightedness}). As in the proof of the previous theorem we may assume that $|X_i| \geq 2$ and $|Y_i|\le 3$ for all $i=1,\ldots,k$. We note that there must exist $i\in [k]$ such that $|X_i| = 2$. If this does not hold, then $|X_i| \geq 3$ for all $i$ and the multisets $X_{1}\cup \ldots \cup X_{k}\cup P$ and $Y_{1}\cup \ldots\cup Y_{k}\cup \emptyset$ cannot be equal. Similarly, there must exist $j\in [k]$ for which $|Y_j| = 3$.
Now by Theorem~\ref{2_proper} in the proper case and its corollary in the strong case we deduce that $G$ is roughly weighted.
\end{proof}

The following example shows that the requirement for the game to be strong or proper cannot be discarded.

\begin{example}[Game with five players that is not roughly weighted]
\label{5-13}
We define the game $G=(P,W)$, where $P=[5]$,  by defining the set of minimal winning coalitions to be
\[
W^m=\{\{1,2\}, \{3,4,5\}\}.
\]
Then the trading transform
\begin{equation}
\label{certificate}
{\cal T}=\{\{1,2\}^5, \{3,4,5\}^7, P; \{2,3,5\}^4, \{2,3,4\}^2, \{1,3,4\}^2, \{1,4,5\}^4,\emptyset\}.
\end{equation}
is a potent certificate of non-weightedness. Indeed, all four coalitions $ \{2,3,5\}$, $\{2,3,4\}$, $\{1,3,4\}$, $\{1,4,5\}$ are losing since they do not contain $\{1,2\}$ or  $\{3,4,5\}$. 
\end{example}

We note that any simple constant sum game with five players is weighted \cite{vNM}.

\begin{theorem}
\label{rough_weights_for 6}
Any simple constant sum game with six players  is roughly weighted.
\end{theorem}

\begin{proof}
Assume on the contrary, that $G$ doesn't have rough weights. Then there is a a potent certificate of non-weightedness of the form (\ref{certificate_of_non-weightedness}). 
Because $G$ is proper, by Theorem~\ref{2_proper} every $X_i$ has at least three elements.
As $G$ is also strong, so by Corollary~\ref{n-2_strong} every $Y_j$ contains at most three elements.
However, it is impossible to find such trading transform $\cal T$ under these constraints since multisets $X_{1}\cup \ldots \cup X_{k}\cup P$ and $Y_{1}\cup \ldots\cup Y_{k}\cup \emptyset$ have different cardinalities.
\end{proof}

\begin{example}[Proper game with six players that is not roughly weighted]
We define $G=(P,W),$ where $P=[6].$ Let the set of minimal winning coalitions be
\[ W^m=\{ \{ 1,2,3 \}, \{ 3,4,5 \},\{ 1,5,6 \}, \{2,4,6 \} , \{ 1,2,6 \} \}.\]
A potent certificate of  non-weightedness for this game is 
\[{\cal T} =  (\{1,2,3\},\{3,4,5\}^2,\{1,5,6\},\{2,4,6\},\{1,2,6\},P; \{1,2,4,5\}^2, \{1,3,4,6\}^2, \{2,3,5,6\}^2, \emptyset ).\]

\end{example}

\section{Conclusion and further research}

In this paper we proved a criterion of existence of rough weights in the language of trading transforms similar to the criterion of existence of ordinary weights given by Taylor and Zwicker \cite{TZ92}.  We defined two functions $f(n)$ and $g(n)$ which measure the deviation of a simple game from a weighted majority game and roughly weighted majority game, respectively. We improved the upper and lower bounds for $f(n)$ and obtain upper and lower bounds for $g(n)$.

\section*{Appendix}

\begin{proof}[Proof of Theorem~\ref{ram1theorem}]
In one direction the statement is clear: if a linear combination  (\ref{ram1}) with coefficients $\row rm$ exists, then there are no non-negative solutions for $A{\bf x}\ge {\bf 0}$, other than ${\bf x}={\bf 0}$. Indeed,  suppose ${\bf s}=(\row sn)$ is such a solution. Denote ${\bf r}=(\row rm)$. If ${\bf x}=(\row xn)\ge {\bf 0}$ is a non-zero  solution, then  (\ref{ram1}) implies $({\bf r}A){\bf x}=s_1x_1+\ldots+s_nx_n < 0$, which is impossible  since ${\bf r}(A{\bf x})\ge 0$.

Let us prove the reverse statement by induction.  For $n=1$ the matrix $A$ is an $m\times 1$ matrix system reduces to $a_{11}x_1\ge 0,\ldots, a_{m1}x_1\ge 0$. Since it has no positive solutions we have $a_{i1}<0$ for some $i$. Suppose $a_{i1}=-\frac{s}{t}$, where $s,t$ are integers and the fraction $\frac{s}{t}$ is in lowest possible terms. Then we can take $r_i=t$ and $r_j=0$ for  $j\ne i$ and obtain $\blcomb ram=ta_{i1}=-s<0$ and $s$ is an integer.

Suppose now that the statement is proved for all $m\times k$ matrices $A$ with $k<n$. Suppose now $A$ is an $m\times n$ matrix and  the system $A{\bf x}\ge{\bf 0}$ has no non-negative solutions other than ${\bf x}={\bf 0}$.

Suppose that a certain column, say the $j$th one, has no positive coefficients. Then we may drop this column and the resulting system will still have no non-zero non-negative solutions (otherwise we can take it, add $x_j=0$, and obtain a non-zero non-negative solution for the original system). By the induction hypothesis for the reduced system we can find non-negative integers $\row rm$ and a vector ${\bf u}$ such that (\ref{ram1}) is true. Then the same $\row rm$ will work also for the original system.

We may now assume that any variable has positive coefficients.  Let us consider the variable $x_1$. Its coefficients are not all negative but they are not all positive either (otherwise the system would have a non-zero non-negative solution ${\bf x}=(1,0,\ldots,0)$). Multiplying, if necessary, the rows of $A$ by positive rational numbers, we find that our system is equivalent to a system of the form
\begin{eqnarray*}
x_1- f_i(\srow xn)&\ge &0,\\
-x_1+g_j(\srow xn)&\ge &0,\\
h_p(\srow xn)&\ge &0,
\end{eqnarray*}
$i=1,\ldots,k$, $j=1,\ldots,m$, $p =1,\ldots,\ell $,
where $f_i(\srow xn)$, $g_j(\srow xn)$ and $h_p(\srow xn)$ are linear functions  in $\srow xn$, and $k\ge 1$, $m\ge 1$.
The matrix of such system has the rows:
${\bf U}_i=(1, {\bf u}_i)$,
${\bf V}_j=(-1, {\bf v}_j)$,
${\bf W}_s=(0, {\bf w}_s)$, where $i=1,\ldots,k,\ j=1,\ldots,m$, $s=1,\ldots, \ell$.
Then the following  system of $km+m+\ell$ inequalities
\begin{eqnarray*}
g_i(\srow xn)&\ge & f_j(\srow xn),\\
g_i(\srow xn)&\ge & 0,\\
h_s(\srow xn)&\ge &0,
\end{eqnarray*}
$i=1,\ldots,k,\ j=1,\ldots,m$, $s=1,\ldots, \ell$,  has no non-negative solutions other than $x_2=\ldots=x_n=0$.  Indeed, if such a solution $(x_2,\ldots, x_n)$ is found then we can set $x_1=\min g_i(\srow xn)$ and since this minimum is non-negative to obtain a non-zero solution of the original inequality.

By the induction hypothesis  there exist non-negative integers $a_{i,j}$, $d_j$, $e_t$, $ s$, where $i=1,\ldots,k,\ j=1,\ldots,m$, $t=1,\ldots, \ell$ such that at least one of these integers positive and
\begin{equation}
\label{aijds}
\sum_{j=1}^m\sum_{i=1}^k a_{i,j}({\bf u}_i+{\bf v}_j)+
\sum_{j=1}^m d_j {\bf v}_j
+\sum_{t=1}^{\ell}e_t{\bf w}_t+{\bf s}_1={\bf 0},
\end{equation}
where ${\bf s}\in \R^{n-1}$ has all its coefficients non-negative.
Let $B_i=\sum_{j=1}^ma_{i,j}$, $C_j=\sum_{i=1}^ka_{i,j}$, $d=\sum_{j=1}^m d_j$. Then $\sum_{i=1}^k B_i=\sum_{j=1}^m C_j$, and (\ref{aijds}) can be rewritten as
\begin{equation}
\label{ijcs}
\sum_{j=1}^k B_i{\bf U}_i+\sum_{j=1}^m (C_j+d_j){\bf V}_j+\sum_{t=1}^{\ell}e_t{\bf W}_t+d{\bf e}_1+(0,{\bf s}_1)={\bf 0}.
\end{equation}
If $d\ne 0$, then we have finished the proof. If not, then we found the numbers $r^1_1,\ldots, r^1_m$ such that
\begin{equation}
\label{ram2}
r^1_1{\bf a}_1+\ldots + r^1_m{\bf a}_m +(0,{\bf s}_1)={\bf 0}.
\end{equation}
Similarly, we can find the numbers $r^2_1,\ldots, r^2_m$ such that
\begin{equation}
\label{ram3}
r^2_1{\bf a}_1+\ldots+ r^2_m{\bf a}_m +({\bf s}_2, 0)={\bf 0}.
\end{equation}
But then we can set $r_i=r^1_1+r^2_1$, $i=1,\ldots, m$ and ${\bf s}=(0,{\bf s}_1)+({\bf s}_2, 0)$ and obtain (\ref{ram1}).  Now all coordinates of ${\bf s}$ are positive, hence the statement is proved.
\end{proof}


\begin{thebibliography}{99}

\bibitem{CF} F. Carreras and J. Freixas. (1996) Complete simple games. Mathematical Social Sciences,
32(2):139Ð155.

\bibitem{CS} Conder, M. and 
Slinko, A. (2004) A counterexample to Fishburn's
conjecture on finite 
linear qualitative probability. Journal of Mathematical
Psychology {\bf 48}, 
425--431.

\bibitem{Elgot} Elgot, C.C. (1961) Truth Functions Realizable by Single Threshold Organs, AIEE
Conference Paper 60-1311 (Oct. 1960), revised Nov. 1960; Switching Circuit Theory
and Logical Design, Sept. (1961), 341-345.

\bibitem{PF1} Fishburn, P.C. (1996)
Finite Linear Qualitative Probability,
{\it Journal of Mathematical Psychology} 40, 64--77.

\bibitem{PF2} Fishburn, P.C. (1997)
Failure of Cancellation Conditions for Additive Linear
Orders, {\it Journal of Combinatorial Designs} 5, 353--365.

\bibitem{FM}  Freixas, J., and Molinero, X. (2009) Simple games and weighted games: A theoretical and computational viepoint. Discrete Applied Mathematics {\bf 157}, 1496--1508.

\bibitem{Gabelman} Gabelman, I.J. (1961) The functional behavior of majority (threshold) elements. Ph.D. diss. Electrical Engineering Department, Syracuse University.

\bibitem{Gale} Gale, D. (1960) The Theory of Linear Economic Models. McGraw-Hill.

\bibitem{Hd} Hadamard, J. (1893) R\'{e}solution d\`{u}ne question relative aux d\'{e}terminats. {\it Bull. Sci. Math.}, 2, 240--246.

\bibitem{HS} Huang, Y. and Schmidt B. (2008) Uniqueness of some cyclic projective planes. Des. Codes Cryptogr.
DOI 10.1007/s10623-008-9229-z

\bibitem{428} H. Kharaghani and B. Tayfeh-Rezaie, A Hadamard matrix of order 428, 2004. \url{http://math.ipm.ac.ir/tayfeh-r/papersandpreprints/h428.pdf}

\bibitem{KPS} Kraft, C.H., Pratt, J.W., and Seidenberg, A.
(1959).  Intuitive Probability on Finite Sets,
{\it Annals of Mathematical Statistics} 30, 408--419.

\bibitem{SM} Marshall, S. (2007) On the Existence of Extremal Cones and Comparative Probability Orderings.  Journal of Mathematical Psychology {\bf 51(5)}, 319-324.

\bibitem{Hall} Marshall Hall, Jr. (1986) Combinatorial Theory, Second Edition, A Wiley-Interscience Publication John Wiley and Sons.

\bibitem{Mu} Muroga, S. (1971) Threshold logic and Its Applications. Wiley Interscience, New York.

\bibitem{Rich} Richardson, M. (1956) On finite projective games. Proc. Amer. Math. Soc., {\bf 7}, 458--465.

\bibitem{Slinko} Slinko, A. (2009) Additive Representability of Finite Measurement Structures. In: "The Mathematics of Preference, Choice, and Order: Essays in Honor of Peter C. Fishburn". Eds. SJ Brams, WV Gehrlein and FS Roberts.  Springer, 2009.

\bibitem{Syl} Sylvester, J.J. (1867) Thoughts on inverse orthogonal matrices, simultaneous sign successions, and tessellated pavements in two or more colours, with applications to Newton's rule, ornamental tile-work, and the theory of numbers. Philosophical Magazine, 34:461-475.

\bibitem{TZ92} Taylor, A.D., and Zwicker, W.S. (1992) A characterization of weighted voting. {\it Proceedings of the American Mathematical Society} {\bf 115}, 1089--1094

\bibitem{TZ95} Taylor, A.D., and Zwicker, W.S. (1995) Simple Games and Magic Squares. Journal of Combinatorial Theory, Series A 71, 67-88.

\bibitem{TZ} Taylor, A.D., and Zwicker, W.S. (1999) Simple games. Princeton University Press. Princeton. NJ

\bibitem{vNM} von Neumann, J., and Morgenstern, O. (1944) Theory of games and economic behavior. Princeton University Press. Princeton. NJ

\bibitem{Sh} Shapley, L.S., (1962) Simple Games: An Outline of the Descriptive Theory. Behavioral Science Vol. 7, pp. 59-66.

\bibitem{W} Winder, R.O. (1962) Threshold Logic. PhD Thesis. Department of Mathematics, Princeton University.

\end{thebibliography}
\end{document}